\documentclass{article}

\usepackage[square,sort&compress,numbers]{natbib}
\usepackage[australian]{babel}
\usepackage{fullpage}

\usepackage{hyperref}

\usepackage{enumitem}
\setenumerate[0]{label=(\arabic*)}
\setlist{leftmargin=*}
\setlist{topsep=1pt plus 0pt minus 1pt,partopsep=0pt,itemsep=1pt plus 0pt minus 1pt,parsep=0.1\parskip}

\usepackage{amssymb,amsmath,amsfonts}
\usepackage{amsthm}
\usepackage[capitalise]{cleveref}

\allowdisplaybreaks

\usepackage{xcolor}
\definecolor{MyBlue}{cmyk}{1,0.13,0,0.63}
\definecolor{MyGreen}{cmyk}{0.91,0,0.88,0.52}
\newcommand{\mylinkcolor}{MyBlue}
\newcommand{\mycitecolor}{MyGreen}
\newcommand{\myurlcolor}{webbrown}

\makeatletter
\def\@endtheorem{\endtrivlist}% NEW
\makeatother
\theoremstyle{plain}
\newtheorem{thm}{Theorem}[section]
\newtheorem{lem}[thm]{Lemma}
\newtheorem{prop}[thm]{Proposition}

\theoremstyle{definition}
\newtheorem{defn}[thm]{Definition}
\newtheorem{remark}[thm]{Remark}

\renewcommand{\eqref}[1]{\labelcref{#1}}
\crefname{thm}{Theorem}{Theorems}
\crefname{lem}{Lemma}{Lemmas}
\crefname{prop}{Proposition}{Propositions}
\crefname{coro}{Corollary}{Corollaries}
\crefname{defn}{Definition}{Definitions}
\crefname{example}{Example}{Examples}
\crefname{remark}{Remark}{Remarks}

\hypersetup{%
  bookmarksnumbered=true,bookmarksopen=false,%
  plainpages=false,% necessary to prevent duplicate page identifiers
  linktocpage=true,%
  colorlinks=true,breaklinks=true,%
  linkcolor=\mylinkcolor,citecolor=\mycitecolor,urlcolor=\myurlcolor,%
  pdfpagelayout=OneColumn,%
  pageanchor=true,%
}

\RequirePackage{slashed}
\newcommand{\sD}{\slashed{D}}

\newcommand{\Z}{\mathbb{Z}}
\newcommand{\R}{\mathbb{R}}
\newcommand{\C}{\mathbb{C}}

\newcommand{\mH}{\mathcal{H}}

\newcommand{\D}{\mathcal{D}}
\newcommand{\A}{\mathcal{A}}
\newcommand{\E}{\mathcal{E}}

\newcommand{\mL}{\mathcal{L}}
\newcommand{\mJ}{\mathcal{J}}

\DeclareMathOperator{\End}{End}
\DeclareMathOperator{\Dom}{Dom}

\newcommand{\bundlefont}[1]{#1}%{{\mathtt{#1}}}
\newcommand{\bS}{\bundlefont{S}}

\newcommand{\la}{\langle}
\newcommand{\ra}{\rangle}

\newcommand{\into}{\hookrightarrow}

\renewcommand{\Re}{\mathop{\textnormal{Re}}}
\renewcommand{\Im}{\mathop{\textnormal{Im}}}
\renewcommand{\bar}[1]{\overline{#1}}
\newcommand{\mvert}{\,|\,}
\newcommand{\bigmvert}{\,\big|\,}

\newcommand{\mattwo}[4]{
  \left(\!\!\!\begin{array}{c@{~}c}#1&#2\\#3&#4\\\end{array}\!\!\!\right)
}

\title{
\textit{Addendum to}\\%
Indefinite Kasparov modules and pseudo-Riemannian manifolds%
}
\author{
Koen van den Dungen%
\footnote{Email: \texttt{kdungen@uni-bonn.de}}
\\[4mm]
{\normalsize 
Mathematisches Institut}, 
{\normalsize Universit\"at Bonn}\\
{\normalsize Endenicher Allee 60, D-53115 Bonn}\\[2mm]
}

\begin{document}
\date{}
\maketitle

\begin{abstract}
We improve our previous results on indefinite Kasparov modules, which provide a generalisation of unbounded Kasparov modules modelling non-symmetric and non-elliptic (e.g.\ hyperbolic) operators. 
In particular, we can weaken the assumptions that are imposed on indefinite Kasparov modules. Using a new theorem by Lesch and Mesland on the self-adjointness and regularity of the sum of two weakly anticommuting operators, we show that we still have an equivalence between indefinite Kasparov modules and pairs of Kasparov modules. Importantly, the weakened version of indefinite Kasparov modules now includes the main motivating example of the Dirac operator on a pseudo-Riemannian manifold. 
The appendix contains a construction of an approximate identity for weakly commuting operators, which is due to Lesch and Mesland. 

\vspace{\baselineskip}

\end{abstract}

\section{Introduction}

In a previous paper \cite{vdDR16}, we presented a definition of indefinite Kasparov modules, providing a generalisation of unbounded Kasparov modules modelling non-symmetric and non-elliptic (e.g.\ hyperbolic) operators. 
Our main theorem showed that to each indefinite Kasparov module we can associate a pair of (genuine) Kasparov modules, and that this process is reversible. 
The main assumption we imposed in the definition of an indefinite Kasparov module $(\A,E_B,\D)$ is that $\Re\D$ and $\Im\D$ almost anticommute. This means, roughly speaking, that the anticommutator $\{\Re\D,\Im\D\}$ is relatively bounded by $\Re\D$. The main tool we used is a theorem by Kaad and Lesch \cite[Theorem 7.10]{KL12}, which states that the sum of two almost anticommuting operators is regular and self-adjoint. 
The purpose of this short note is to improve on the results presented in \cite{vdDR16}. 

The main issue with the results of \cite{vdDR16} is that, unfortunately, our main motivating example, namely the Dirac operator $\sD$ on a pseudo-Riemannian 
manifold, does not (in general) satisfy the definition of an indefinite Kasparov module. For such a Dirac operator, the real part $\Re\sD$ contains the spacelike derivatives, while the imaginary part $\Im\sD$ contains the timelike derivatives. The assumption that $\Re\sD$ and $\Im\sD$ almost anticommute then means that the anticommutator $\{\Re\sD,\Im\sD\}$ contains only spacelike derivatives. In general, however, this anticommutator is a first-order differential operator containing both spacelike and timelike derivatives. 

Thus, in order to improve our results, we need a generalisation of the theorem by Kaad and Lesch, in which the anticommutator $\{\Re\D,\Im\D\}$ is allowed to be relatively bounded by the `combined graph norm' of $\Re\D$ and $\Im\D$. This generalisation is now available thanks to recent work by Lesch and Mesland \cite{LM19}. 
Hence we can weaken the assumptions that we imposed on indefinite Kasparov modules, while the main results in \cite{vdDR16} remain valid. 
The main advantage, for our purposes, is that the weakened definition of indefinite Kasparov modules is naturally satisfied by a Dirac operator on a pseudo-Riemannian manifold (under only mild conditions), since $\{\Re\sD,\Im\sD\}$ is always a first-order differential operator. 

In Section \ref{sec:weakly_anticommuting}, we will review the results of Lesch and Mesland \cite{LM19}. 
Moreover, we will show that the sum of two weakly commuting (instead of anticommuting) operators is also essentially self-adjoint (though in general not closed). 
The proof of this fact relies on an alternative method of proof for the main result of \cite{LM19}, which is also due to Lesch and Mesland, and which is included in the appendix. 
In Section \ref{sec:Dirac_decomp}, we will describe a natural example of weakly anticommuting operators, by decomposing the Dirac operator on a Riemannian spin manifold with a given orthogonal direct sum decomposition of the tangent bundle. 
Finally, in Section \ref{sec:indefinite}, we will show that we can weaken the assumptions in the definitions given in \cite{vdDR16}, while all the results of \cite{vdDR16} remain valid.

\subsection*{Acknowledgements}

The author thanks Matthias Lesch and Bram Mesland for interesting discussions, helpful suggestions, and for kindly allowing the author to reproduce their proof of Proposition \ref{prop:approx_id} in the Appendix. 
The author also thanks the referees, whose comments have led to significant improvements in the article.

\section{Weakly (anti)commuting operators}
\label{sec:weakly_anticommuting}

We consider regular self-adjoint operators $S$ and $T$ on a Hilbert $B$-module $E$ (where $B$ is a $C^*$-algebra), such that $\Dom S\cap\Dom T$ is dense. For $x,y\in\Dom S\cap\Dom T$, we define the \emph{`combined graph inner product'}
\[
\la x | y \ra_{S,T} := \la x|y\ra + \la Sx|Sy\ra + \la Tx|Ty\ra ,
\]
and denote the corresponding norm by $\|\cdot\|_{S,T}$. 

We denote by $[S,T]_\pm$ the (anti)commutator $ST\pm TS$, which is defined on the natural \emph{initial} domain 
\[
\Dom([S,T]_\pm) = \big\{ x\in\Dom S\cap\Dom T : Sx\in\Dom T \;\&\; Tx\in\Dom S \big\} .
\]

Rather than the notion of almost (anti)commuting operators given in \cite[Definition 2.8]{vdDR16} (which was based on \cite[Assumption 7.1]{KL12}), we will now consider the following (weaker) notion of weakly (anti)commuting operators. 

\begin{defn}[{\cite[Definition 2.1]{LM19}}]
\label{defn:weakly_(anti)commuting}
Two regular self-adjoint operators $S$ and $T$ on a Hilbert $B$-module $E$ are called \emph{weakly (anti)\-com\-muting} if 
\begin{enumerate}
\item there is a constant $C>0$ such that for all $x\in\Dom([S,T]_\pm)$ we have 
\[
\big\la [S,T]_\pm x \mvert [S,T]_\pm x \big\ra \leq C \la x|x\ra_{S,T} ;
\]
\item there is a core $\E\subset\Dom T$ such that $(S+\lambda)^{-1}(\E) \subset \Dom[S,T]_\pm$ for $\lambda\in i\R$, $|\lambda|\geq\lambda_0>0$. 
\end{enumerate}
\end{defn}

Note that, although the formulation of the second assumption is asymmetric in $S$ and $T$, it follows a posteriori that the assumption is also satisfied with $S$ and $T$ interchanged \cite[\S3]{LM19}. Moreover, the second assumption also holds with $\E = \Dom T$ \cite[Proposition 3.5]{LM19}. 
The main result of \cite{LM19} is the following:

\begin{thm}[{\cite[Theorem 2.6]{LM19}}]
\label{thm:sum_sa}
Let $S$ and $T$ be weakly anticommuting operators on a Hilbert $B$-module $E$. 
Then the operator $S+T$ is regular and self-adjoint on the domain $\Dom(S+T) = \Dom S\cap\Dom T$. 
\end{thm}

This theorem can be proved in (at least) two different ways. 
One method is based on the following proposition, which is also due to Matthias Lesch and Bram Mesland (but is not included in \cite{LM19}). The proof of this proposition is given in the Appendix. 

\begin{prop}
\label{prop:approx_id}
Let $S$ and $T$ be weakly commuting operators on a Hilbert $B$-module $E$. 
Then $\lambda^2 \big[ S , (S+\lambda)^{-1} (T+\lambda)^{-1} \big]_-$ and $\lambda^2 \big[ T , (S+\lambda)^{-1} (T+\lambda)^{-1} \big]_-$ are uniformly bounded (for $\lambda\in i\R$, $|\lambda|\geq\lambda_0$) and converge strongly to zero as $\lambda\to\pm i\infty$. 
\end{prop}

Given two weakly commuting operators $S$ and $T$, one may check that $S\pm iT$ are closed operators. Thanks to Proposition \ref{prop:approx_id}, we have an approximate identity $A_n := -n^2 (S-in)^{-1} (T-in)^{-1}$ such that $[S\pm iT,A_n]$ converges strongly to zero as $n\to\infty$. By a standard argument (analogous to the proof of Proposition \ref{prop:sum_ess_sa} below) it then follows that $(S\pm iT)^* = (S\mp iT)$. Using a doubling trick (\emph{cf.}\ \cite[\S2.4]{LM19}), this result can then be translated to the self-adjointness of the sum $S+T$ in case of two weakly \emph{anti}commuting operators $S$ and $T$. Finally, one can apply the local-global principle \cite{Pie06,KL12} to prove regularity of $S+T$. 

The proof of Theorem \ref{thm:sum_sa} given in \cite{LM19} is different, and in fact proves a stronger statement. Indeed, the proof in \cite{LM19} not only shows that $S+T$ is regular self-adjoint, but also that the resolvent $(S+T+\mu)^{-1}$ (with $\mu\in i\R$) can be approximated by $(S+T+\lambda^{-1}ST+\mu)^{-1}$ as $|\lambda|\to\infty$ ($\lambda\in i\R$). 

The advantage of the method via Proposition \ref{prop:approx_id} is that it also allows us to prove that the sum of two weakly \emph{commuting} operators (instead of anticommuting operators) is \emph{essentially} self-adjoint. (Note that the sum of weakly commuting operators is in general not closed; the obvious example is $T=-S$.) 

\begin{prop}
\label{prop:sum_ess_sa}
Let $S$ and $T$ be weakly commuting operators on a Hilbert $B$-module $E$. 
Then $S+T$ is essentially self-adjoint on $\Dom S\cap\Dom T$. 
\end{prop}
\begin{proof}
Since $S+T$ is symmetric, it suffices to prove that $\Dom(S+T)^* \subset \Dom(\bar{S+T})$. 
Let $\xi\in\Dom(S+T)^*$. 
Using the approximate identity $A_n := -n^2 (S-in)^{-1} (T-in)^{-1}$, we define the sequence 
\[
\xi_n := A_n \xi \in \Dom(S+T) ,
\]
which converges in norm to $\xi$ (see Lemma \ref{lem:approx_id_1}). For $\eta\in\Dom(S+T)$, we can calculate
\begin{align*}
\la\xi_n|(S+T)\eta\ra 
&= \big\la A_n \xi \bigmvert (S+T) \eta \big\ra 
= \big\la \xi \bigmvert A_n^* (S+T) \eta \big\ra \\
&= \big\la \xi \bigmvert (S+T) A_n^* \eta \big\ra - \big\la \xi \bigmvert \big[ S+T , A_n^* \big] \eta \big\ra \\
&= \big\la A_n (S+T)^* \xi \bigmvert \eta \big\ra - \big\la \big[ S+T , A_n^* \big]^* \xi \bigmvert \eta \big\ra .
\end{align*}
Hence we find
\begin{align}
\label{eq:sum_approx}
(S+T)\xi_n = (S+T)^*\xi_n 
&= A_n (S+T)^* \xi - \big[ S+T , A_n^* \big]^* \xi .
\end{align}
By Lemma \ref{lem:approx_id_1}, the first term converges strongly to $(S+T)^*\xi$. 
Furthermore, on $\Dom(S+T)$ we have the equality 
\[
\big[ S+T , A_n^* \big]^* = - \big[ S+T , A_n \big] .
\]
Since both sides of this equality are bounded and adjointable, the left-hand-side equals the closure of the right-hand-side on all of $E$. Hence we know from Proposition \ref{prop:approx_id} that the second term in Eq.\ \eqref{eq:sum_approx} converges to zero. Thus $(S+T)\xi_n$ converges, which proves that $\xi\in\Dom(\bar{S+T})$. 
\end{proof}

Again, one can try to apply the local-global principle to prove regularity of (the closure of) $S+T$. However, since we will not need regularity in the remainder of this article, we will not pursue this any further.

\section{A decomposition of the Dirac operator}
\label{sec:Dirac_decomp}

In this section we will describe a decomposition of the Dirac operator as a sum of two weakly anti-commuting operators. 
Let $(M,g)$ be an $n$-dimensional Riemannian spin manifold with the spinor bundle $\bS$. 
Suppose that we have an \emph{orthogonal} decomposition $TM = E_1\oplus E_2$, where $E_1$ and $E_2$ are oriented subbundles of ranks $n_1$ and $n_2$, respectively. 
Locally, we can consider oriented orthonormal frames $\{e_1,\ldots,e_n\}$ of $TM$ such that $e_j\in\Gamma^\infty(E_1)$ for $1\leq j\leq n_1$ and $e_j\in\Gamma^\infty(E_2)$ for $n_1+1\leq j\leq n_1+n_2=n$. 

We consider the Clifford representation $\gamma\colon\Gamma^\infty(TM)\to\Gamma^\infty(\End(\bS))$ (our conventions are such that $\gamma(v)^2=-g(v,v)$ and $\gamma(v)^*=-\gamma(v)$). The Dirac operator $\sD$ on $\Gamma_c^\infty(\bS)$ is given locally by
\[
\sD = \sum_{j=1}^n \gamma(e_j) \nabla^\bS_{e_j} .
\] 
We define a self-adjoint unitary operator $\Gamma_1 \in \Gamma^\infty(\End(\bS))$ on $L^2(S)$ which is locally given by
\[
\Gamma_1 := i^{n_1(n_1+1)/2} \gamma(e_1)\cdots\gamma(e_{n_1}) , 
\]
where $\{e_1,\ldots,e_{n_1}\}$ is a local oriented orthonormal frame of $E_1$. 
We note that, in the case of a pseudo-Riemannian manifold, if the metric is negative-definite on $E_1$ and positive-definite on $E_2$, then the operator $(-i)^{n_1} \Gamma_1$ is the usual \emph{fundamental symmetry} which turns the Hilbert space $L^2(\bS)$ into a Krein space (see \cite[\S4.1]{vdDR16}). 
Next, we define the operators
\begin{align*}
\sD_1 &:= \frac12 \big( \sD-(-1)^{n_1}\Gamma_1\sD\Gamma_1 \big) , & 
\sD_2 &:= \frac12 \big( \sD+(-1)^{n_1}\Gamma_1\sD\Gamma_1 \big) .
\end{align*}
Then $\sD_1$ and $\sD_2$ are both symmetric operators on $\Gamma_c^\infty(\bS)$, and we have $\sD_1+\sD_2=\sD$. In terms of a local orthonormal frame (corresponding to the decomposition $TM = E_1\oplus E_2$), we have the explicit expressions
\begin{align*}
\sD_1 &= \sum_{j=1}^{n_1} \Big( \gamma(e_j) \nabla^\bS_{e_j} - \frac12 \gamma(e_j) \big[ \nabla^\bS_{e_j} , \Gamma_1 \big] \Gamma_1 \Big) + \sum_{k=n_1+1}^n \frac12 \gamma(e_k) \big[ \nabla^\bS_{e_k} , \Gamma_1 \big] \Gamma_1 , \\
\sD_2 &= \sum_{j=1}^{n_1} \frac12 \gamma(e_j) \big[ \nabla^\bS_{e_j} , \Gamma_1 \big] \Gamma_1 + \sum_{k=n_1+1}^n \Big( \gamma(e_k) \nabla^\bS_{e_k} - \frac12 \gamma(e_k) \big[ \nabla^\bS_{e_k} , \Gamma_1 \big] \Gamma_1 \Big) .
\end{align*}

\begin{prop}
\label{prop:diff_weak_anti-comm}
Assume that $M$ is complete and has bounded geometry. Then (the closures of) the operators $\sD_1$ and $\sD_2$ are self-adjoint and weakly anticommuting. 
\end{prop}
\begin{proof}
The completeness of $M$ implies that the symmetric operators $\sD_1$ and $\sD_2$ are self-adjoint. 
The assumption of bounded geometry ensures that the coefficients of $\sD_1$ and $\sD_2$ are globally bounded. 
Since both $\sD_1+\sD_2$ and $\sD_1-\sD_2$ are elliptic, there exists $C>0$ such that for any $\psi\in\Dom\sD_1\cap\Dom\sD_2$ we have  
\[
\|\psi\| + \| (\sD_1\pm\sD_2) \psi \| \leq \|\psi\| + \|\sD_1\psi\| + \|\sD_2\psi\| \leq C \|\psi\|_{\sD_1\pm\sD_2} . 
\]
where $\|\cdot\|_{\sD_1\pm\sD_2}$ denotes the graph norm of $\sD_1\pm\sD_2$. Thus the graph norm of $\sD_1\pm\sD_2$ is equivalent to the combined graph norm $\|\cdot\|_{\sD_1,\sD_2}$, and we have the equality $\Dom\sD_1\cap\Dom\sD_2 = \Dom(\sD_1\pm\sD_2)$. 
Since the principal symbols of $\sD_1$ and $\sD_2$ anticommute, we know that the anticommutator $[\sD_1,\sD_2]_+$ is a first-order differential operator, which (by bounded geometry) has globally bounded coefficients. 
Using again ellipticity of $\sD_1\pm\sD_2$ and the equivalence $\|\cdot\|_{\sD_1\pm\sD_2} \sim \|\cdot\|_{\sD_1,\sD_2}$, there exists a $C'>0$ such that for all $\psi \in \Dom[\sD_1,\sD_2]_\pm$ we have the inequality 
\[
\big\| [\sD_1,\sD_2]_+ \psi \big\| \leq C' \|\psi\|_{\sD_1,\sD_2} . 
\]
It follows that condition (1) of Definition \ref{defn:weakly_(anti)commuting} is satisfied. 

To prove the domain condition (2), we will make use of a clever Clifford matrix trick, inspired by the proof of \cite[Theorem 5.1]{LM19}. Consider the operators 
\begin{align*}
S &:= \mattwo{0}{\sD_1}{\sD_1}{0} , & 
T &:= \mattwo{\sD_1+\sD_2}{0}{0}{\sD_2-\sD_1} .
\end{align*}
We note that the combined graph norm $\|\cdot\|_{S,T}$ is equivalent to the graph norm of the elliptic operator $T$, and the anticommutator 
\[
[S,T]_+ = \mattwo{0}{{}[\sD_1,\sD_2]_+}{{}[\sD_1,\sD_2]_+}{0} 
\]
is relatively bounded by $T$. Hence $S$ and $T$ satisfy condition (1) of Definition \ref{defn:weakly_(anti)commuting}. 
Furthermore, using again the assumption of bounded geometry, we note that $\Dom T$ is a core for $S$, and that the second-order differential operators $ST$ and $TS$ are well-defined on the domain $(T\pm i)^{-1} \cdot \Dom T = \Dom T^2$ of the second-order \emph{elliptic} operator $T^2$. So $S$ and $T$ also satisfy condition (2) of Definition \ref{defn:weakly_(anti)commuting}. 
Thus $S$ and $T$ weakly anticommute, and it follows from \cite[Theorem 2.6.(2)]{LM19} that we also have the domain condition
\[
(S\pm i)^{-1} \cdot \Dom T = \Dom[S,T]_\pm .
\]
Rephrasing this in terms of $\sD_1$ and $\sD_2$, we obtain
\begin{multline*}
(\sD_1\pm i)^{-1} \cdot (\Dom\sD_1\cap\Dom\sD_2) \\
= \left\{ \psi\in\Dom\sD_1\cap\Dom\sD_2 : \sD_1\psi\in\Dom\sD_1\cap\Dom\sD_2 , \; \sD_2\psi \in \Dom\sD_1 \right\} . 
\end{multline*}
Since $\Dom\sD_1\cap\Dom\sD_2$ is a core for $\sD_2$, this proves that $\sD_1$ and $\sD_2$ also satisfy condition (2) of Definition \ref{defn:weakly_(anti)commuting}. 
\end{proof}

\section{Indefinite Kasparov modules}
\label{sec:indefinite}

First, let us briefly recall our notion of (reverse) `Wick rotations' \cite[Definitions 2.2 \& 2.5]{vdDR16}. Given a closed operator $\D$ on a Hilbert $B$-module $E$ such that $\Dom\D\cap\Dom\D^*$ is dense, we define the \emph{real and imaginary parts} of $\D$ as the closures of 
\begin{align*}
\Re\D &:= \frac12 (\D+\D^*) , & \Im\D &:= -\frac i2 (\D-\D^*) ,
\end{align*}
on the initial domain $\Dom\D\cap\Dom\D^*$. 
Furthermore, we define the \emph{`Wick rotations'} of $\D$ as the closures of 
\begin{align*}
\D_+ &:= \Re\D + \Im\D , & \D_- &:= \Re\D - \Im\D ,
\end{align*}
on the initial domain $\Dom\Re\D\cap\Dom\Im\D$. 

Conversely, given two closed symmetric operators $\D_1$ and $\D_2$ on $E$ such that $\Dom\D_1\cap\Dom\D_2$ is dense, we define the \emph{reverse Wick rotation} of the pair $(\D_1,\D_2)$ as the closure of
$$
\D := \frac12 (\D_1+\D_2) + \frac i2 (\D_1-\D_2) 
$$
on the initial domain $\Dom\D_1\cap\Dom\D_2$. 

We now replace our former definitions of indefinite Kasparov modules \cite[Definition 3.1]{vdDR16} and pairs of Kasparov modules \cite[Definition 3.6]{vdDR16}, using the weaker notion of weakly anticommuting operators described above. 

\begin{defn}
\label{defn:indef_Kasp_mod}
Given (separable) $\Z_2$-graded $C^*$-algebras $A$ and $B$, an \emph{indefinite} unbounded Kasparov $A$-$B$-module $(\A,E_B,\D)$ is given by
\begin{itemize}
\item a $\Z_2$-graded, countably generated, right Hilbert $B$-module $E$;
\item a $\Z_2$-graded $*$-homomorphism $\pi\colon A\to\End_B(E)$;
\item a separable dense $*$-subalgebra $\A\subset A$;
\item a closed odd operator $\D\colon\Dom\D\subset E\to E$ such that 
\begin{enumerate}
\item there exists a linear subspace $\E\subset\Dom\D\cap\Dom\D^*$ which is dense with respect to $\|\cdot\|_{\D,\D^*}$, and which is a core for both $\D$ and $\D^*$; 
\item the operators $\Re\D$ and $\Im\D$ are regular and essentially self-adjoint on $\E$; 
\item the operators $\Re\D$ and $\Im\D$ are \emph{weakly anticommuting};
\item we have the inclusion $\pi(\A)\cdot\E\subset\Dom\D\cap\Dom\D^*$, and the graded commutators $[\D,\pi(a)]_\pm$ and $[\D^*,\pi(a)]_\pm$ are bounded on $\E$ for each $a\in\A$;
\item the map $\pi(a)\circ\iota\colon \Dom\D\cap\Dom\D^*\into E\to E$ is compact for each $a\in A$, where $\iota\colon \Dom\D\cap\Dom\D^*\into E$ denotes the natural inclusion map, and $\Dom\D\cap\Dom\D^*$ is considered as a Hilbert $B$-module with the inner product $\la\cdot|\cdot\ra_{\D,\D^*}$.
\end{enumerate}
\end{itemize}
If $B=\C$ and $A$ is trivially graded, we will write $E=\mH$ and refer to $(\A,\mH,\D)$ as an even \emph{indefinite spectral triple} over $A$. 
\end{defn}

\begin{remark}
In contrast with \cite[Definition 3.1]{vdDR16}, we no longer assume that $\D$ is regular (i.e.\ that $1+\D^*\D$ has dense range), since this assumption is not used anywhere. 
\end{remark}

\begin{defn}
\label{defn:even_pair}
We say $(\A,E_B,\D_1,\D_2)$ is a \emph{pair of unbounded Kasparov $A$-$B$-modules} if $(\A,E_B,\D_1)$ and $(\A,E_B,\D_2)$ are unbounded Kasparov $A$-$B$-modules such that:
\begin{enumerate}
\item there exists a linear subspace $\E\subset\Dom\D_1\cap\Dom\D_2$ which is a common core for $\D_1$ and $\D_2$;
\item the operators $\D_1+\D_2$ and $\D_1-\D_2$ are regular and essentially self-adjoint on $\E$; 
\item the operators $\D_1+\D_2$ and $\D_1-\D_2$ are weakly anticommuting. 
\end{enumerate}
If $B=\C$ and $A$ is trivially graded, we will write $E=\mH$ and refer to $(\A,\mH,\D_1,\D_2)$ as an even \emph{pair of spectral triples} over $A$. 
\end{defn}

Using Theorem \ref{thm:sum_sa} instead of \cite[Corollary 2.12]{vdDR16}, the proof of \cite[Proposition 3.8]{vdDR16} carries through. Furthermore, using Proposition \ref{prop:sum_ess_sa} instead of \cite[Proposition 2.13]{vdDR16}, the proof of \cite[Proposition 3.9]{vdDR16} also carries through. Thus we still have an equivalence between indefinite Kasparov modules and pairs of Kasparov modules. 

\begin{thm}[cf.\ {\cite[Theorem 3.11]{vdDR16}}]
The procedure of (reverse) Wick rotation implements a bijection between indefinite unbounded Kasparov $A$-$B$-modules $(\A,E_B,\D)$ and pairs of unbounded Kasparov $A$-$B$-modules $(\A,E_B,\D_1,\D_2)$. 
This bijection also descends to the corresponding unitary equivalence classes.
\end{thm}

The main advantage of the new version of indefinite Kasparov modules is that the definition now incorporates any pseudo-Riemannian manifold (with only mild assumptions). 

\begin{prop}
Let $(M,g)$ be an $n$-dimensional time- and space-oriented pseudo-Riemannian spin manifold of signature $(t,s)$, with a given spinor bundle $\bS\to M$. 
Let $r$ be a spacelike reflection, such that the associated Riemannian metric $g_r$ is complete. 
Assume furthermore that $(M,g,r,\bS)$ has bounded geometry (as in \cite[Definition 4.1]{vdDR16}). 
Then the canonical Dirac operator $\sD$ on $\bS\to M$ yields an indefinite spectral triple $(C_c^\infty(M), L^2(\bS), \sD)$. 
\end{prop}
\begin{proof}
The main thing to check is that $\Re\sD$ and $\Im\sD$ weakly anticommute (\emph{cf.}\ \cite[\S4.1.2]{vdDR16}), which is done using the same arguments as in Section \ref{sec:Dirac_decomp}. 
Indeed, using a local oriented orthonormal frame $\{e_j\}$ corresponding to the decomposition $TM = E_t\oplus E_s$ as the direct sum of a timelike and a spacelike subbundle (of rank $t$ and $s=n-t$, respectively), 
and defining the fundamental symmetry $\mJ_M := i^{t(t-1)/2} \gamma(e_1)\cdots\gamma(e_t)$, 
we have the explicit expressions \cite[\S4.1.2]{vdDR16}
\begin{align*}
\Re\sD &= \sum_{j=t+1}^n \gamma(e_j) \nabla^\bS_{e_j} + \frac12 \sum_{j=1}^n \gamma(e_j) \mJ_M \big[\nabla^\bS_{e_j},\mJ_M\big] , \\
\Im\sD &= i \sum_{j=1}^t \gamma(e_j) \nabla^\bS_{e_j} + \frac i2 \sum_{j=1}^n \gamma(e_j) \mJ_M \big[\nabla^\bS_{e_j},\mJ_M\big].
\end{align*}
We observe that $\Re\sD$ and $\Im\sD$ are symmetric operators whose principal symbols anti-commute, and that $\Re\sD+\Im\sD$ and $\Re\sD-\Im\sD$ are both elliptic. Hence the same argument as in the proof of Proposition \ref{prop:diff_weak_anti-comm} applies, and it follows that $\Re\sD$ and $\Im\sD$ are self-adjoint and weakly anti-commuting. 
\end{proof}

\newpage
\appendix

\section{An approximate identity for weakly commuting operators}

\begin{center}
\textbf{Based on unpublished notes by Matthias Lesch and Bram Mesland}
\end{center}
\vspace{\baselineskip}

This Appendix contains the proof of Proposition \ref{prop:approx_id}.
The argument is due to Matthias Lesch and Bram Mesland, and the author kindly thanks them for their permission to reproduce their proof here. 

To remind ourselves, and for the convenience of the reader, we recall the following facts:
\begin{itemize}
\item It is a consequence of the Banach-Steinhaus Theorem that a strongly convergent sequence of bounded operators on a Banach space is uniformly norm bounded. 
\item Given a uniformly bounded sequence $(A_n) \subset \mL(X)$ of operators on a Banach space $X$, then for $(A_n)$ being strongly continuous it suffices to show pointwise convergence on a dense subspace. 
\item As a consequence of uniform boundedness, if $(A_n)$ converges strongly to $A$ and if $(B_n)$ converges strongly to $B$, then $(A_n\cdot B_n)$ converges strongly to $A\cdot B$. 
\end{itemize}
These facts will be used repeatedly without mentioning.

Now let $S$ and $T$ be weakly commuting, regular self-adjoint operators on $E$, and denote by $[S,T] = [S,T]_- = ST-TS$ the ordinary commutator. Let $\lambda,\mu\in i\R$ with $|\lambda|,|\mu|\geq\lambda_0$. 

\begin{lem}
\label{lem:approx_id_1}
The operators $\lambda (S+\lambda)^{-1}$, $\lambda (T+\lambda)^{-1}$, and $\lambda^2 (S+\lambda)^{-1} (T+\lambda)^{-1}$ converge strongly to the identity as $|\lambda|\to\infty$. 
\end{lem}
\begin{proof}
The family $\big(S(S+\lambda)^{-1}\big)_{|\lambda|\geq\lambda_0}$ is uniformly bounded, and for $\psi\in\Dom S$ we have $\|S(S+\lambda)^{-1} \psi\| = \|(S+\lambda)^{-1} S\psi\| \leq \frac1{|\lambda|} \|S\psi\|$. Hence $S(S+\lambda)^{-1}$ converges to zero strongly. Consequently, $\lambda(S+\lambda)^{-1} = 1 - S(S+\lambda)^{-1}$ converges strongly to the identity. 
Thus the product $\lambda (S+\lambda)^{-1} \lambda (T+\lambda)^{-1}$ also converges strongly to the identity. 
\end{proof}

\begin{lem}
\label{lem:approx_id_2}
For $\lambda_0$ large enough and for $\lambda,\mu\in i\R$ with $|\lambda|,|\mu|>\lambda_0$, 
the operator families $[S,T] (S+\lambda)^{-1} (T+\mu)^{-1}$ and $[S,T] (T+\lambda)^{-1} (S+\mu)^{-1}$ converge strongly to zero as $|\lambda|\to\infty$ (for fixed $\mu$). 
\end{lem}
\begin{proof}
Since $S(S+\lambda)^{-1}$ converges strongly to zero and $(T+\mu)^{-1}$ is bounded, we know that also $S (S+\lambda)^{-1} (T+\mu)^{-1}$ converges strongly to zero as $|\lambda|\to\infty$. We will show that also $T (S+\lambda)^{-1} (T+\mu)^{-1}$ converges strongly to zero as $|\lambda|\to\infty$. We write
\begin{align}
\label{eq:converge_zero}
T (S+\lambda)^{-1} (T+\mu)^{-1} &= (S+\lambda)^{-1} [S,T] (S+\lambda)^{-1} (T+\mu)^{-1} \nonumber\\
&\quad+ (S+\lambda)^{-1} T (T+\mu)^{-1} .
\end{align}
Since $T(T+\mu)^{-1}$ is bounded, the second summand is of order $|\lambda|^{-1}$ and therefore converges in norm to zero. For the first summand, we note that by \cite[Lemma 3.2]{LM19} we have for $\lambda_0$ large enough and $\psi\in\Dom[S,T]$ that there exists $c>0$ such that 
\[
\big\| [S,T] \psi \big\| \leq c \left(\frac1{|\lambda|}+\frac1{|\mu|}\right) \big\| (T+\mu) (S+\lambda) \psi \big\| .
\]
In particular, $[S,T] (S+\lambda)^{-1} (T+\mu)^{-1}$ is uniformly bounded, so the first summand in Eq.\ \eqref{eq:converge_zero} also converges strongly to zero. 
Finally, by condition (1) of Definition \ref{defn:weakly_(anti)commuting} we have 
\begin{align*}
\big\| [S,T] (S+\lambda)^{-1} (T+\mu)^{-1} \psi \big\|^2 
\leq C^2 &\Big( \big\| (S+\lambda)^{-1} (T+\mu)^{-1} \psi \big\|^2 \\
&\quad+ \big\| S (S+\lambda)^{-1} (T+\mu)^{-1} \psi \big\|^2 
+ \big\| T (S+\lambda)^{-1} (T+\mu)^{-1} \psi \big\|^2 \Big) ,
\end{align*}
which shows that $[S,T] (S+\lambda)^{-1} (T+\mu)^{-1} \psi$ converges strongly to zero as well. 
Interchanging $S$ and $T$, the same result also applies to $[S,T] (T+\lambda)^{-1} (S+\mu)^{-1}$. 
\end{proof}

\begin{lem}
\label{lem:approx_id_3}
The operators $\lambda^2 \big[ S , (S+\lambda)^{-1} (T+\lambda)^{-1} \big]$ and $\lambda^2 \big[ T , (S+\lambda)^{-1} (T+\lambda)^{-1} \big]$ are uniformly bounded for $|\lambda|\geq\lambda_0$. 
\end{lem}
\begin{proof}
We write 
\[
\lambda^2 \big[ T , (S+\lambda)^{-1} (T+\lambda)^{-1} \big] = \lambda^2 (S+\lambda)^{-1} [S,T] (S+\lambda)^{-1} (T+\lambda)^{-1} .
\]
The first factor $\lambda^2 (S+\lambda)^{-1}$ is of order $|\lambda|$ in norm. 
The second factor $[S,T] (S+\lambda)^{-1} (T+\lambda)^{-1}$ is of order $|\lambda|^{-1}$ by \cite[Lemma 3.2]{LM19}. Hence $\lambda^2 \big[ T , (S+\lambda)^{-1} (T+\lambda)^{-1} \big]$ is uniformly bounded (in norm) for $|\lambda|\geq\lambda_0$. 
By interchanging $S$ and $T$, we see that also $\lambda^2 \big[ S , (T+\lambda)^{-1} (S+\lambda)^{-1} \big]$ is uniformly bounded (in norm). It then follows that 
\[
\lambda^2 \big[ S , (S+\lambda)^{-1} (T+\lambda)^{-1} \big] = - \left( \bar\lambda^2 \big[ S , (T+\bar\lambda)^{-1} (S+\bar\lambda)^{-1} \big] \right)^*
\]
is uniformly bounded as well. 
\end{proof}

\begin{proof}[\textbf{Proof of Proposition \ref{prop:approx_id}}]
By Lemma \ref{lem:approx_id_3} it suffices to establish strong convergence on a dense submodule of $E$. 
With $\lambda_0$ as in Lemma \ref{lem:approx_id_2}, and for some $\mu\in i\R$ with $|\mu|>\lambda_0$, we rewrite
\begin{align*}
\lambda^2 \big[ T , (S+\lambda)^{-1} (T+\lambda)^{-1} \big] 
= \lambda (S+\lambda)^{-1} \cdot [S,T] (S+\lambda)^{-1} (T+\mu)^{-1} \cdot \lambda (T+\lambda)^{-1} \cdot (T+\mu) .
\end{align*}
As $|\lambda|\to\infty$, the first and third factors converge strongly to the identity by Lemma \ref{lem:approx_id_1}, while the second factor converges strongly to zero by Lemma \ref{lem:approx_id_2}. Thus this proves that $\lambda^2 \big[ T , (S+\lambda)^{-1} (T+\lambda)^{-1} \big]$ converges strongly to zero on the dense submodule $\Dom T$, and hence on $E$. 
Next, we rewrite 
\begin{align*}
\lambda^2 \big[ S , (S+\lambda)^{-1} (T+\lambda)^{-1} \big] 
= \lambda^2 (S+\lambda)^{-1} (T+\lambda)^{-1} \cdot [T,S](T+\lambda)^{-1} (S+\mu)^{-1} \cdot (S+\mu) .
\end{align*}
Again, by Lemmas \ref{lem:approx_id_1} and \ref{lem:approx_id_2}, as $|\lambda|\to\infty$ the first factor converges strongly to the identity, while the second factor converges strongly to zero. Thus $\lambda^2 \big[ S , (S+\lambda)^{-1} (T+\lambda)^{-1} \big]$ converges strongly to zero on $\Dom S$, and hence on $E$. 
\end{proof}

\newcommand{\MR}[1]{}

%\bibliographystyle{myamsalpha}
%\bibliography{short,bibliography}

\providecommand{\noopsort}[1]{}\providecommand{\vannoopsort}[1]{}
\providecommand{\bysame}{\leavevmode\hbox to3em{\hrulefill}\thinspace}
\providecommand{\MR}{\relax\ifhmode\unskip\space\fi MR }
% \MRhref is called by the amsart/book/proc definition of \MR.
\providecommand{\MRhref}[2]{%
  \href{http://www.ams.org/mathscinet-getitem?mr=#1}{#2}
}
\providecommand{\href}[2]{#2}

\end{document}